\documentclass[11pt]{article}
\usepackage{amsmath,amsthm,amssymb,ifsym,a4wide} 
\usepackage[twoside]{geometry}
\usepackage{graphicx}
\usepackage{epsfig} 
\usepackage{slashed}  
\usepackage[hidelinks]{hyperref}
\newtheorem{theorem}{Theorem}[section]
\newtheorem{proposition}[theorem]{Proposition}
\newtheorem{lemma}[theorem]{Lemma}

\newtheorem{remark}[theorem]{Remark} 
\numberwithin{equation}{section} 
\numberwithin{figure}{section}  
\usepackage{amssymb, amscd, mathrsfs, wasysym} 

\newcommand \la \langle
\newcommand \ra \rangle

\newcommand \Kcal {\mathcal{K}}
\newcommand \Hcal {\mathcal{H}}

\newcommand \underdel {\underline \partial}
\newcommand \overdel {\overline \partial}

\newcommand \trianglerightNEW \triangleright

\newcommand \auth {\textsc}

\newcommand \bei {\begin{itemize}}
\newcommand \eei {\end{itemize}}
\newcommand \be {\begin{equation}}
\newcommand \bel {\be\label}
\newcommand \ee {\end{equation}}
\newcommand \del \partial
\newcommand \RR {\mathbb R}

\newcommand \eps \epsilon

\usepackage[most]{tcolorbox} 

\let\oldmarginpar\marginpar
\renewcommand\marginpar[1]{\-\oldmarginpar[\raggedleft\footnotesize #1]%
{\raggedright\footnotesize #1}}
\begin{document}

\title{\bf \Large 
Stability of a wave and Klein-Gordon system\\ with mixed coupling} 

\author{Shijie Dong
\footnote{Fudan University, School of Mathematical Sciences, Shanghai, China. 
Email: dongs@ljll.math.upmc.fr, shijiedong1991@hotmail.com.
}}


\date{\today}
\maketitle

\begin{abstract} 
{
We are interested in establishing stability results for a system of semilinear wave and Klein-Gordon equations with mixed coupling nonlinearities, that is, we consider all of the possible quadratic nonlinear terms of the type of wave and Klein-Gordon interactions. The main difficulties are due to the absence of derivatives on the wave component in the nonlinearities. By doing a transformation on the wave equation, we reveal a hidden null structure. Next by using the scaling vector field on the wave component only, which was generally avoided, we are able to get very good $L^2$--type estimates on the wave component. Then we distinguish high order and low order energies of both wave and Klein-Gordon components, which allows us to close the bootstrap argument.
}
\end{abstract}
{\sl MSC code.} 35L05, 35L52, 35L71.
\newline
{\sl Keywords.} Wave and Klein-Gordon system; global existence; sharp pointwise decay; null condition; hyperboloidal foliation method.

\tableofcontents

\section{Introduction} 
\paragraph{Model problem.}

We are interested in studying the following coupled wave and Klein-Gordon system

\bel{eq:model}
\aligned
-\Box u 
&= Q_{u0}(u; v, \del v) + Q_{u1}(\del u; v, \del v),
\\
-\Box v + v 
&= Q_{v0}(u; v, \del v) + Q_{v1}(\del u; v, \del v), 
\\
Q_{u0}(u; v, \del v) = M_u u v + & M_u^{\alpha} u \del_\alpha v,
\qquad
Q_{u1}(\del u; v, \del v) = N_u^\alpha \del_\alpha u v + N_u^{\alpha \beta} \del_\alpha u \del_\beta v,
\\
Q_{v0}(u; v, \del v) = M_v u v + & M_v^{\alpha} u \del_\alpha v,
\qquad
Q_{v1}(\del u; v, \del v) = N_v^\alpha \del_\alpha u v + N_v^{\alpha \beta} \del_\alpha u \del_\beta v,
\endaligned
\ee
whose initial data are prescribed on the time slice $t=t_0$ (we always take $t_0 = 2$)
\bel{eq:model-ID}
\aligned
\big( u, \del_t u \big) (t_0, \cdot)
&=
\big( u_0, u_1 \big),
\\
\big( v, \del_t v \big) (t_0, \cdot)
&=
\big( v_0, v_1 \big).
\endaligned
\ee
In the above, $M_u, M_v, M_u^{\alpha}, M_v^{\alpha}, N_u^\alpha, N_v^\alpha, N_u^{\alpha \beta}, N_v^{\alpha \beta}$ are constants, and no null conditions are assumed. Throughout, Roman letters take values within $\{0, 1, 2, 3\}$ and Latin letters take values in $\{1, 2, 3\}$, and Einstein summation convention is adopted unless otherwise mentioned. The wave operator is denoted by $\Box = \eta^{\alpha \beta} \del_\alpha \del_\beta$, with the metric $\eta = \text{diag} (-1, 1, 1, 1)$. We note that the system \eqref{eq:model} takes into account all of the possible semilinear nonlinearities of wave and Klein-Gordon interactions.

Motivated by the work \cite{Georgiev, Katayama12a, Dong1, DLW}, we want to establish the small data global existence result for the model problem \eqref{eq:model}, with physical models (like Dirac-Proca model etc.) behind.
Recall the remark in \cite{Dong1} that the main difficulty is that the nonlinearities are critical. That is, even if the solution $(u, v)$ behaves very nicely, we still have the following critical non-integrable quantities
\be 
\aligned
\big\| Q_{u0}(u; v, \del v) + Q_{u1}(\del u; v, \del v) \big\|_{L^2(\RR^3)}
&\sim t^{-1},
\\
\big\| Q_{v0}(u; v, \del v) + Q_{v1}(\del u; v, \del v) \big\|_{L^2(\RR^3)}
&\sim t^{-1},
\endaligned
\ee
which are due to the fact that in some terms in the nonlinearities of \eqref{eq:model} there are no derivatives on the wave component $u$. Since the standard energy for the wave equation cannot directly control norms like $\| u \|_{L^2(\RR^3)}$, we regard the presence of $u$ in the nonlinearities as a bad indication.

This problem was proposed in \cite{PLF-YM-book}, where the authors were not sure whether the nonlinearties $u v, u \del v$ lead to global solutions or finite time blowup, and then was partially solved in \cite{Dong1} with positive result. In this article, we will also be able to deal with the nonlinearities of the type $u \del v$ in the Klein-Gordon equation, which was not treated in \cite{Dong1}.

\paragraph{Brief history.} We first very briefly revisit some previous work in the study of pure wave or pure Klein-Gordon equations (mainly in $\RR^{3+1}$), and then recall the history on the study of coupled wave and Klein-Gordon systems. 
In the work by John \cite{John2, John}, it was proved that general quadratic nonlinearities in the wave equations cannot ensure global-in-time solutions. Later on, Klainerman \cite{Klainerman86} and Christodoulou \cite{Christodoulou} independently proved that wave equations admit global-in-time solutions if the nonlinearities satisfy null condition. One refers to \cite{L-R-cmp, L-R-annals} and \cite{P-S-cpam} for the study on the generalisation of the null condition.
As for the Klein-Gordon equations, we only recall here the breakthrough by Klainerman \cite{Klainerman85} and Shatah \cite{Shatah}, where the authors independently showed that general quadratic nonlinearities can guarantee the existence of small global solutions to the Klein-Gordon equations. Thus it is natural to ask what kind of quadratic nonlinearities lead to global solutions to the coupled wave and Klein-Gordon equations?

The study of the coupled wave and Klein-Gordon systems has attracted people's attention since decades ago. It is very important to study such coupled systems, not only because they are of great interest in the view of pure PDE, but also for the reason that the coupled wave and Klein-Gordon systems can be derived from many important physical models, like the Dirac-Proca model and Klein-Gordon-Zakharov model, the Maxwell-Klein-Gordon model, the Einstein-Klein-Gordon model and many others.
As far as we know, the study on this subjects was initiated by Bachelot \cite{Bachelot} on the Dirac-Klein-Gordon system. Later on Georgiev \cite{Georgiev} proved the global existence result for the nonlinearities satisfying the strong null condition (i.e. $\del_\alpha u \del_\beta v - \del_\beta u \del_\alpha v$, $\alpha, \beta \in \{0, 1, 2, 3\}$), which are compatible with the vector fields excluding the scaling vector field.
Then Tsutsumi and his collaborators proved Klein-Gordon-Zakharov system in \cite{OTT}, Dirac-Proca system in \cite{Tsutsumi}, and the Maxwell-Higgs system in \cite{Tsutsumi2}. In the similar time, Psarelli studied the Maxwell-Klein-Gordon equations in \cite{PsarelliA, PsarelliT}, and as far as we know this is the first time people considered bad (with on derivaties on the wave component $u$) nonlinearities of the type $u \del v$. Long time after that, Katayama \cite{Katayama12a, Katayama18} proved a large class of coupled wave and Klein-Gordon equations. In \cite{Katayama12a}, for example, Katayama considered the bad nonlinearities of the type $u v, u \del v$, but with the assumption that the nonlinearities in the wave equation of $u$ have divergence form, and this means that the bad nonlinearities are essentially of the type $\del u v, \del u \del v$, which are not very bad. Later on, LeFloch and Ma also tackled a large class of the coupled wave and Klein-Gordon equations with compactly supported initial data in \cite{PLF-YM-book} using the hyperboloidal foliation method, which is essentially the Klainerman's vector field method on hyperboloids. Soon after that, LeFloch and Ma proved the nonlinear stability result of the Einstein-Klein-Gordon system in \cite{PLF-YM-cmp}, where they considered the bad nonlinearities of the type $u \del \del v$. Later on, Wang \cite{Wang} and Ionescu-Pausader \cite{Ionescu-P} also proved global existence of the Einstein-Klein-Gordon system. Recently, part of the nonlinearities of the type $u v, u \del v$ was considered in \cite{Dong1} for a coupled wave and Klein-Gorodn system. Besides, we also mention the recent work \cite{PLF-YM-arXiv1, PLF-YM-arXiv2, Klainerman-QW-SY, Fang, Ionescu-P2} on coupled wave and Klein-Gordon equations, where no compactness assumption on the initial data are needed.

\paragraph{Main result.}


Our goal is to prove the existence of global-in-time solution to the model problem \eqref{eq:model}, and to further investigate the pointwise decay result of the solution. 

\begin{theorem}[Nonlinear stability of the system of wave and Klein-Gordon equations with compact support]
\label{thm:wKG}
Consider the system \eqref{eq:model} with initial data supported in $\{ (t_0=2, x) : |x| \leq 1 \}$
and let $N \geq 7$ be an integer.  
Then there exists $\eps_0 > 0$, and the initial value problem \eqref{eq:model}--\eqref{eq:model-ID} admits a global-in-time solution $(u, v)$ for
all compactly supported initial data $(u_0, u_1, v_0, v_1)$
satisfying the smallness condition 
\bel{main-thm-initial-data}
\| u_0, v_0 \|_{H^{N+1}(\RR^3)} + \| u_1, v_1 \|_{H^N(\RR^3)} 
\leq \eps, 
\qquad
\eps \in (0, \eps_0).
\ee
Furthermore, we have the following sharp pointwise decay results
\bel{eq:thm}
| u(t, x) | 
\lesssim t^{-1} (t-|x|)^{-1/2},
\quad
| v(t, x) |
\lesssim t^{-3/2}.
\ee 
\end{theorem}

Note that the nonlinearities $Q_{u1}(\del u; v, \del v), Q_{v1}(\del u; v, \del v)$ are regarded as good ones, and the main difficulties are caused by the nonlinearities $Q_{u0}(u; v, \del v), Q_{v0}(u; v, \del v)$ in which there are no derivatives on the $u$ part.   
Recall that in \cite{Dong1}, the nonlinearities $Q_{u0}(u; v, \del v)$ can be treated by doing a transformation and revealing the hidden null structure of the nonlinearities $Q_{u0}(u; v, \del v)$ which is inspired by the work of Tsutsumi in \cite{Tsutsumi}. As for the $uv$ term appearing in $Q_{v0}(u; v, \del v)$ of the Klein-Gordon equation, the novel idea in \cite{Dong1} is that we move this term to the left hand side and regard it as a small pertubation of the mass of the Klein-Gordon equation, and then apply the techniques in \cite{PLF-YM-cmp} to treat the Klein-Gordon equation with mass $m = \sqrt{1-u}$. But the $u \del v$ term in the $Q_{v0}(u; v, \del v)$  cannot be treated in \cite{Dong1}.

Due to the point of view that the null form $\del^\alpha u \del_\alpha v$ is not consistent with the Klein-Gordon equations, the existing ways of handling it do not seem to be precise enough. On one hand, the treatment on $\del^\alpha u \del_\alpha v$ was missed and the notion of strong null form $\del_\alpha u \del_\beta v - \del_\beta u \del_\alpha v$ was come up with in \cite{Georgiev}, while on the other hand, the way using the semi-hyperboloidal frame to estimate $\del^\alpha u \del_\alpha v$, introduced in \cite{PLF-YM-book}, can be inteprated as treating the product of two Klein-Gordon components. Neither of the above ways could provide a good factor of $t^{-1}$ as in the pure wave equation cases. Fortunately, we find that we can still gain some extra decay by treating the null forms $\del^\alpha u \del_\alpha v$ in the classical manner, which is the key to treating the $u \del v$ term in the Klein-Gordon equation compared to \cite{Dong1}, and which relies on the following observations: (1) in the estimate of 
$$
\del^\alpha u \del_\alpha v
=
{1\over t} \big( \del_t v L_0 u - L_a v \del^a u \big),
$$
we are able to move the scaling vector field $L_0$ to the wave component $u$ only, and we take $L^2$--type estimate on $L_0 u$ when estimating the low order energies of the null form $\del^\alpha u \del_\alpha v$;
(2) the conformal energy estimates allow us to bound the $L^2$--type norm of the term $L_0 u$;
(3) we are able to obtain pointwise decay results of $u$ and $L_0 u$ by avoiding using the scaling vector field $L_0$ (twice).

Of independent interest, we find a surprising phenomena, but let us recall some existing results before we provide the statement. 
In \cite{Alinhac, PLF-YM-cmp}, the authors proved independently the following pointwise decay results for wave equations, which can be roughly described as: let $w$ solve the wave equation
\be 
\aligned
- &\Box w = f,
\\
w(t_0, x) &= \del_t w(t_0, x) =0,
\endaligned
\ee
with $f$ spatially compactly supported and satisfying
\be 
| f | \leq C_f t^{- 2 - \nu} (t - r)^{-1 + \mu},
\ee
for $0<\mu \leq 1/2$ and $0<\nu \leq 1/2$ with $C_f$ some constant, and then one has
\be 
| w(t, x) | \lesssim {C_f \over \nu \mu} (t - r)^{\mu - \nu} t^{-1}.
\ee
This result shows that if the homogeneity (total power) of the decay rate of the source term $f$ is $-3 - \nu + \mu$, then the homogeneity of the decay rate of the solution $w$ is $-1 - \nu + \mu$.
Now we are ready to state the surprising phenomena: recall the nonlinearities $Q_{u0}(u; v, \del v) + Q_{u1}(\del u; v, \del v)$ in the wave equation of $u$, and the best pointwise decay we can expect for it is 
$$
\big| Q_{u0}(u; v, \del v) + Q_{u1}(\del u; v, \del v) \big|
\sim t^{-5/2} (t - |x|)^{-1/2},
$$
whose homogeneity is $-3$, but we can prove in Theorem \ref{thm:wKG} the following strong pointwise decay of $u$ 
$$
| u(t, x) | 
\lesssim t^{-1} (t-|x|)^{-1/2},
$$
whose homogeneity is $-3/2$ instead of $-1 = -3 + 2$ which is what we predict.


\begin{remark}
In Theorem \ref{thm:wKG} we assumed that the initial data are compactly supported, which is due to the use the hyperboloidal foliation of the spacetime. But we believe this restriction can be removed either by the ways used in \cite{PLF-YM-arXiv1, PLF-YM-arXiv2, Klainerman-QW-SY}, or by using the Sobolev inequalities proposed in \cite{Klainerman2} and \cite{Georgiev2} to get the pointwise decay for the wave and Klein-Gordon components respectively.

\end{remark}



\subsection*{Outline}

The rest of this article is organised as follows.

In Section \ref{sec:BHFM}, we revisit some preliminaries on the wave equations, the hyperboloidal foliation method, and some important inequalities. Then we prove Theorem \ref{thm:wKG} relying on the bootstrap method in Section \ref{sec:BA}.


\section{Preliminaries}
\label{sec:BHFM}
 
\subsection{Hyperboloidal foliation of Minkowski spacetime}

We first recall some basic notations in \cite{PLF-YM-book} about the hyperboloidal foliation method,
so that we can introduce the energy functional for wave or Klein-Gordon components on hyperboloids. We work in the $(3+1)$-- dimensional Minkowski spacetime, whose signature is taken to be $(-, +, +, +)$. We write a point $(t, x) = (x^0, x^1, x^2, x^3)$ in Cartesion coordinates, and its spatial radius is denoted by $r := | x | = \sqrt{(x^1)^2 + (x^2)^2 + (x^3)^2}$. The partial derivatives are denoted by 
$$
\del_\alpha = \del_{x^\alpha},
\qquad
\alpha=0, 1, 2, 3,
$$
and might be abbreviated to $\del$, while the Lorentz boosts are represented by
$$
L_a := x^a \del_t + t \del_a, \qquad a= 1, 2, 3,
$$
and might be abbreviated to $L$.
Throughout, all functions are supported in the interior of the future light cone $\Kcal:= \{(t, x): r< t-1 \}$ with vertex $(1, 0, 0, 0)$. The hyperboloidal hypersurfaces are denoted by $\Hcal_s:= \{(t, x): t^2 - r^2 = s^2 \}$ with hyperbolic time $s\geq 2$. We note that for any point $(t, x)$ within the cone $\Kcal$, the following relation holds
$$
r < t,
\qquad
s < t < s^2.
$$


The semi-hyperboloidal frame we will use is defined by
\bel{eq:semi-hyper}
\underdel_0:= \del_t, \qquad \underdel_a:= {L_a \over t} = {x^a\over t}\del_t+ \del_a.
\ee
The vectors $\underdel_a$ are tangent to the hyperboloids. Besides we also note that the vector field $\underdel_\perp:= \del_t+ (x^a / t)\del_a$ is orthogonal to the hyperboloids.
The semi-hyperboloidal frame and the natural Cartesian frame can be transformed to each other by the following relations
\be 
\aligned
\del_t = \underdel_0,
\qquad
\del_a = - {x^a \over t} \del_t + \underdel_a.
\endaligned
\ee

On the other hand, the hyperboloidal frame is defined by
\bel{eq:hyper-frame}
\overdel_0:= \del_s = {s\over t} \del_t, \qquad \overdel_a:= \underdel_a = {x^a\over t}\del_t+ \del_a,
\ee
and its relation with the Cartesian frame reads
\be 
\aligned
\del_t = {t\over s} \del_s,
\qquad
\del_a = - {x^a \over s} \del_s + \overdel_a.
\endaligned
\ee


\subsection{Energy estimates on hyperboloids}

Following \cite{PLF-YM-cmp} and for a function $\phi$ defined on a hyperboloid $\Hcal_s$, we define its energy $E_m$ by
\bel{eq:2energy} 
\aligned
E_m(s, \phi)
&:=
 \int_{\Hcal_s} \Big( \big(\del_t \phi \big)^2+ \sum_a \big(\del_a \phi \big)^2+ 2 (x^a/t) \del_t \phi \del_a \phi + m^2 \phi ^2 \Big) \, dx
\\
               &= \int_{\Hcal_s} \Big( \big( (s/t)\del_t \phi \big)^2+ \sum_a \big(\underdel_a \phi \big)^2+ m^2 \phi^2 \Big) \, dx
                \\
               &= \int_{\Hcal_s} \Big( \big( \underdel_\perp \phi \big)^2+ \sum_a \big( (s/t)\del_a \phi \big)^2+ \sum_{a<b} \big( t^{-1}\Omega_{ab} \phi \big)^2+ m^2 \phi^2 \Big) \, dx,
 \endaligned
 \ee
and in the above $\Omega_{ab}:= x^a\del_b- x^b\del_a$ represent the rotational vector fields, and $\underdel_{\perp}= L_0/t = \del_t+ (x^a / t) \del_a$ is the orthogonal vector field. For simplicity we will denote $E(s, \phi):= E_0(s, \phi)$.
The integral above in $L^1(\Hcal_s)$ is defined by
\bel{flat-int}
\|\phi \|_{L^1_f(\Hcal_s)}
:=\int_{\Hcal_s}|\phi | \, dx 
=\int_{\RR^3} \big|\phi(\sqrt{s^2+r^2}, x) \big| \, dx.
\ee
We remind one that it holds
$$
\big\| (s/t) \del \phi \big\|_{L^2_f(\Hcal_s)} + \sum_\alpha \big\| \overdel_\alpha \phi \big\|_{L^2_f(\Hcal_s)}
+
\big\| t^{-1} L_0 \phi \big\|_{L^2_f(\Hcal_s)} 
\lesssim
E_m(s, \phi)^{1/2}.
$$

Next, we demonstrate the energy estimates to the hyperboloidal setting.

\begin{proposition}[Energy estimates for wave-Klein-Gordon equations]
For $m \geq 0$ and for $s \geq s_0$ (with $s_0 = 2$), it holds that
\bel{eq:w-EE} 
E_m(s, u)^{1/2}
\leq 
E_m(s_0, u)^{1/2}
+ \int_2^s \big\| -\Box u + m^2 u \big\|_{L^2_f(\Hcal_{s'})} \, ds'
\ee
for all sufficiently regular functions $u$, which are defined and supported in $\Kcal_{[s_0, s]}$.
\end{proposition}

One refers to \cite{PLF-YM-cmp} for the proof.


\subsection{Sobolev-type and Hardy-type inequality}

We now state a Sobolev-type inequality adapted to the hyperboloidal setting, which will be used to obtain sup-norm estimates for both wave and Klein-Gordon components. The Sobolev-type inequalities on hyperboloids have been proved by Klainerman \cite{Klainerman85}, H{\"o}rmander \cite{Hormander}, Georgiev \cite{Georgiev3}, and LeFloch-Ma \cite{PLF-YM-book, PLF-YM-cmp}. We will use the version of LeFloch-Ma, in which we only need to use the vector fields of Lorentz boosts (recall that other versions of the Sobolev-type inequality also require the use of the translation vector fields and the rotation vector fields). For its proof, one refers to \cite{PLF-YM-book, PLF-YM-cmp} for details.

\begin{lemma} \label{lem:sobolev}
Let $u= u(t, x)$ be sufficient smooth and be supported in $\{(t, x): |x|< t - 1\}$ and let $s \geq 2$, then it holds 
\bel{eq:Sobolev2}
\sup_{\Hcal_s} \big| t^{3/2} u(t, x) \big|  \lesssim \sum_{| J |\leq 2} \big\| L^J u \big\|_{L^2_f(\Hcal_s)},
\ee
with $L$ the Lorentz boosts and $J$ the multi-index. As a consequence, we also have
\be
\sup_{\Hcal_s} \big| s \hskip0.03cm t^{1/2} u(t, x) \big|  \lesssim \sum_{| J |\leq 2} \big\| (s/t) L^J u \big\|_{L^2_f(\Hcal_s)},
\ee
\end{lemma}

The following Hardy-type inequality adapted to the  hyperboloidal foliation is used to estimate the $L^2$--type of norm for the wave component $u$. For the proof, one refers to \cite{PLF-YM-book} for instance.
\begin{lemma}
For all sufficiently regular functions $u$ defined and supported in the cone $\Kcal = \{(t, x): |x|< t - 1\}$, and for all $s\geq s_0$, one has 
\bel{eq:Hardy} 
\big\| r^{-1} u \big\|_{L^2_f(\Hcal_s)}
\lesssim
\sum_{a} \big\| \underdel_a u \big\|_{L^2_f(\Hcal_s)}.
\ee
\end{lemma}


\subsection{Conformal-type energy estimates on hyperboloids}
\label{subsc:conf}
We now recall a conformal-type energy, introduced in \cite{YM-HH}, which is adapted to the hyperboloidal foliation setting. This lemma allows us to obtain a robust estimate of the $L^2$-type norm for the wave component $u$.

\begin{lemma}\label{lem:conformal}
Define the conformal-type energy for a sufficiently regular function $u$, which is supported in the cone $\Kcal = \{(t, x): |x|< t - 1\}$, by
\be 
E_{con} (s, u)
:=
\int_{\Hcal_s} \Big( \sum_a \big( s \underdel_a u \big)^2 + \big( K u + 2 u \big)^2 \Big) \, dx,
\ee
where we used the notation
$$
K u 
:= \big( s \del_s + 2 x^a \underdel_a \big) u.
$$
Then we have
\bel{eq:con-estimate} 
E_{con} (s, u)^{1/2}
\leq 
E_{con} (s_0, u)^{1/2}
+
2 \int_{s_0}^s s' \big\| \Box u \big\|_{L^2_f(\Hcal_{s'})} \, ds'.
\ee
\end{lemma}

\begin{lemma}\label{lem:L2type}
We have in the cone $\Kcal$ that
\bel{eq:l2type-wave} 
\big\| (s / t) u \big\|_{L^2_f(\Hcal_s)}
+
\big\| (s / t) L_0 u \big\|_{L^2_f(\Hcal_s)}
+
\sum_a \big\| (s / t) L_a u \big\|_{L^2_f(\Hcal_s)}
\lesssim
E_{con} (s, u)^{1/2}.
\ee

\begin{proof}

First, by the definition of the conformal energy it is obvious to see that
$$
\sum_a \big\| (s / t) L_a u \big\|_{L^2_f(\Hcal_s)}
\leq
E_{con} (s, u)^{1/2}.
$$

Second, the Hardy inequality \eqref{eq:Hardy} gives
$$
\big\| (1 / r) u \big\|_{L^2_f(\Hcal_s)}
\lesssim
\sum_a \big\| \underdel_a u \big\|_{L^2_f(\Hcal_s)},
$$
which further yields
$$
\big\| (s / r) u \big\|_{L^2_f(\Hcal_s)}
\lesssim
\sum_a \big\| s \underdel_a u \big\|_{L^2_f(\Hcal_s)}
\lesssim
E_{con} (s, u)^{1/2}.
$$
Since the relation $r < t$ holds within the cone $\Kcal$, we thus obtain
$$
\big\| (s / t) u \big\|_{L^2_f(\Hcal_s)}
\lesssim
E_{con} (s, u)^{1/2}.
$$

In order to take into account the scaling vector field $L_0$, we next express the scaling vector field in the semi-hyperboloidal frame
$$
L_0
=
t \del_t + x^a \del_a
=
s \del_s + x^a \underdel_a.
$$
Recall that
$$
K u 
= \big( s \del_s + 2 x^a \underdel_a \big) u
= L_0 u + x^a \underdel_a u. 
$$
Hence the triangle inequality implies that
$$
\aligned
\big\| (s/t) L_0 u \big\|_{L^2_f(\Hcal_s)}
&\leq
\big\| (s/t) K u \big\|_{L^2_f(\Hcal_s)}
+
\big\| (s/t)  x^a \underdel_a u \big\|_{L^2_f(\Hcal_s)}
\\
&=
\big\| (s/t) K u \big\|_{L^2_f(\Hcal_s)}
+
\big\| (s/t)  (x^a/t) L_a u \big\|_{L^2_f(\Hcal_s)}.
\endaligned
$$
Finally the facts that $s<t, |x^a| < t$ in the cone $\Kcal$ indicates that
$$
\big\| (s/t) L_0 u \big\|_{L^2_f(\Hcal_s)}
\lesssim
E_{con}(s, u)^{1/2}.
$$

The proof is complete.

\end{proof}

\end{lemma}


\subsection{Estimates for commutators and null forms}

We recall the following standard estimates for null forms, which can be found in \cite{Sogge}.

\begin{lemma}\label{lem:null-classical}
It holds
\be 
\del_\alpha u \del^\alpha v
=
-{1\over t} \big( \del_t v L_0 u - L_a v \del^a u \big).
\ee
\end{lemma}

Next, we revisit the following commutation relations from \cite{Sogge}.

\begin{lemma}\label{lem:commu-eq}
For all nice functions $u, v$ we have
\be
\aligned
\del_\alpha \big( \del_\gamma u \del^\gamma v \big)
&=
\del_\gamma \del_\alpha u \del^\gamma v + \del_\gamma u \del^\gamma \del_\alpha v,
\\
L_a \big( \del_\gamma u \del^\gamma v \big)
&=
\del_\gamma L_a u \del^\gamma v + \del_\gamma u \del^\gamma L_a v.
\endaligned
\ee
\end{lemma}

The following estimates for commutators will also be frequently used, and one refers to \cite{Sogge, PLF-YM-book} for the proof.

\begin{lemma}
Within the cone $\Kcal$ it holds for any nice function $u$ that
\be 
\aligned
\big|\del_\alpha L_a u \big|
&\lesssim
\big| L_a \del_\alpha u \big| + \sum_\beta \big| \del_\beta u \big|,
\\
\big| L_a L_b u \big|
&\lesssim
\big| L_b L_a u \big| + \sum_m \big| L_m u \big|,
\\
\big| L_a L_0 u \big|
&\lesssim
\big| L_0 L_a u \big| + \sum_\alpha \big(\big| L_\alpha u \big| + \big| \del_\alpha u \big| \big),
\\
\big| L_a (u \,s/t) \big|
&\lesssim
\big| (s/t) L_a u \big| +  \big| (s/t) u \big|,
\\
\big| L_b L_a (u \,s/t) \big| 
&\lesssim
\big| (s/t) L_b L_a u \big| 
+ \big| (s/t) u \big| + \sum_m \big| (s/t) L_m u \big|.
\endaligned
\ee
\end{lemma}


\section{Bootstrap argument}
\label{sec:BA}

\subsection{Bootstrap assumptions and its consequences}\label{subsec:BA}

Our goal in this section is to prove Theorem \ref{thm:wKG} by relying on a standard bootstrap argument.

According to the local well-posedness result of wave and Klein-Gordon equations and the continuation of the energies, we assume that the following bounds are true for all $s \in [s_0, T)$ (with $s_0 = 2$):
\bel{eq:BA} 
\aligned
E_{con}(s, \del^I L^J u)^{1/2}
&\leq
C_1 \eps s^{1/2 +\delta},
\qquad
&|I| + |J| \leq N-1,
\\
E_{con}(s, \del^I L^J u)^{1/2}
&\leq
C_1 \eps s^\delta,
\qquad
&|I| + |J| \leq N-2,
\\
E_{con}(s, \del^I L^J u)^{1/2}
&\leq
C_1 \eps,
\qquad
&|I| + |J| \leq N-3,
\\
E_1(s, \del^I L^J v)^{1/2} + E(s, \del^I L^J u)^{1/2}
&\leq 
C_1 \eps s^\delta,
\qquad
& |I| + |J| \leq N,
\\
E_1(s, \del^I L^J v)^{1/2} + E(s, \del^I L^J u)^{1/2}
&\leq 
C_1 \eps,
\qquad
& |I| + |J| \leq N-1.
\endaligned
\ee
In the above, the constant $C_1>1$ is some large number to be determined, $\delta>0$ is a small constant, and the time $T$ is defined by
$$
T := \sup \{s > s_0 : \eqref{eq:BA}~holds  \}.
$$

We want to prove the following refined estimates (with $T < + \infty$)
\bel{eq:BA-refined}
\aligned
E_{con}(s, \del^I L^J u)^{1/2}
&\leq
{1\over 2} C_1 \eps s^{1/2 + \delta},
\qquad
&|I| + |J| \leq N-1,
\\
E_{con}(s, \del^I L^J u)^{1/2}
&\leq
{1\over 2} C_1 \eps s^\delta,
\qquad
&|I| + |J| \leq N-2,
\\
E_{con}(s, \del^I L^J u)^{1/2}
&\leq
{1\over 2} C_1 \eps,
\qquad
&|I| + |J| \leq N-3,
\\
E_1(s, \del^I L^J v)^{1/2} + E(s, \del^I L^J u)^{1/2}
&\leq 
{1\over 2} C_1 \eps s^\delta,
\qquad
& |I| + |J| \leq N,
\\
E_1(s, \del^I L^J v)^{1/2} + E(s, \del^I L^J u)^{1/2}
&\leq 
{1\over 2} C_1 \eps,
\qquad
& |I| + |J| \leq N-1.
\endaligned
\ee
which immediately imply the global existence in Theorem \ref{thm:wKG}.

By recalling the Sobolev inequalities and the estimates for commutators, we have the following set of estimates from the bootstrap assumptions \eqref{eq:BA}.

\begin{proposition}\label{prop:BA1}
Under the assumptions in \eqref{eq:BA}, it holds for all $s \in [s_0, T)$ that
\bei

\item $L^2$--type estimates:
\be 
\aligned
&\big\| (s/t) \del \del^I L^J u, (s/t) \del^I L^J \del u, (s/t) \del \del^I L^J v, (s/t) \del^I L^J \del v, \del^I L^J v \big\|_{L^2_f(\Hcal_s)}
\\
\lesssim & C_1 \eps s^\delta,
\qquad
|I| + |J| = N,
\\
&\big\| (s/t) \del \del^I L^J u, (s/t) \del^I L^J \del u, (s/t) \del \del^I L^J v, (s/t) \del^I L^J \del v, \del^I L^J v \big\|_{L^2_f(\Hcal_s)}
\\
\lesssim & C_1 \eps,
\qquad
|I| + |J| \leq N-1.
\endaligned
\ee

\item $L^\infty$--type estimates:
\be 
\aligned
&\big\| s t^{1/2} \del \del^I L^J u, s t^{1/2} \del^I L^J \del u, s t^{1/2} \del \del^I L^J v, s t^{1/2} \del^I L^J \del v, t^{3/2} \del^I L^J v \big\|_{L^\infty(\Hcal_s)}
\\
\lesssim & C_1 \eps s^\delta,
\qquad
|I| + |J| = N-2,
\\
&\big\| s t^{1/2} \del \del^I L^J u, s t^{1/2} \del^I L^J \del u, s t^{1/2} \del \del^I L^J v, s t^{1/2} \del^I L^J \del v, t^{3/2} \del^I L^J v \big\|_{L^\infty(\Hcal_s)}
\\
\lesssim & C_1 \eps,
\qquad
|I| + |J| \leq N-3.
\endaligned
\ee

\eei

\end{proposition}

If we further utilise the estimates in Lemma \ref{lem:L2type}, then we are led to the set of estimates right below.

\begin{proposition}\label{prop:BA2}

We have
\bei

\item $L^2$--type estimates:
\be 
\aligned
&\big\| (s / t) \del^I L^J u \big\|_{L^2_f(\Hcal_s)}
+
\big\| (s / t) L_0 \del^I L^J u \big\|_{L^2_f(\Hcal_s)}
+
\sum_a \big\| (s / t) L_a \del^I L^J u \big\|_{L^2_f(\Hcal_s)}
\\
\lesssim
&C_1 \eps s^{1/2 + \delta},
\qquad
|I| + |J| \leq N-1,
\\
&\big\| (s / t) \del^I L^J u \big\|_{L^2_f(\Hcal_s)}
+
\big\| (s / t) L_0 \del^I L^J u \big\|_{L^2_f(\Hcal_s)}
+
\sum_a \big\| (s / t) L_a \del^I L^J u \big\|_{L^2_f(\Hcal_s)}
\\
\lesssim
&C_1 \eps s^\delta, 
\qquad
|I| + |J| \leq N-2,
\\
&\big\| (s / t) \del^I L^J u \big\|_{L^2_f(\Hcal_s)}
+
\big\| (s / t) L_0 \del^I L^J u \big\|_{L^2_f(\Hcal_s)}
+
\sum_a \big\| (s / t) L_a \del^I L^J u \big\|_{L^2_f(\Hcal_s)}
\\
\lesssim
&C_1 \eps,
\qquad
|I| + |J| \leq N-3.
\endaligned
\ee

\item $L^\infty$--type estimates:
\be 
\aligned
\big\| s t^{1/2} \del^I L^J u \big\|_{L^\infty(\Hcal_s)}
+
\sum_\alpha \big\| s t^{1/2} L_\alpha \del^I L^J u \big\|_{L^\infty(\Hcal_s)}
\lesssim
C_1 \eps,
\qquad
|I| + |J| \leq N-4.
\endaligned
\ee

\eei
\end{proposition}


\subsection{Improved estimates for the wave component $u$}

In this part, we will prove the improved estimates for the wave component $u$, including its energy estimates $E(s, \del^I L^J u)^{1/2}$ as well as its conformal energy estimates $E_{con}(s, \del^I L^J u)^{1/2}$. This part is the main part of Section \ref{sec:BA}.

First, we show the refined estimates for the natural energy of the wave component $u$.
Recall the wave equation in \eqref{eq:model}, and we act $\del^I L^J$ on it to get
\bel{eq:model-del}
-\Box \del^I L^J u 
= \del^I L^J Q_{u0}(u; v, \del v) + \del^I L^J Q_{u1}(\del u; v, \del v),
\ee
and this will be used to obtain improved estimates on $\del^I L^J u$.

\begin{proposition}\label{prop:u-E}
The following refined estimates are true
\be 
\aligned
E(s, \del^I L^J u)^{1/2}
&\lesssim
\eps + (C_1 \eps)^2 s^\delta,
\qquad
 |I| + |J| = N,
\\
E(s, \del^I L^J u)^{1/2}
&\lesssim
\eps + (C_1 \eps)^2,
\qquad
 |I| + |J| \leq N-1.
\endaligned
\ee
\end{proposition}

\begin{proof}

Based on \eqref{eq:model-del} as well as the energy estimates, we have
$$
\aligned
&E(s, \del^I L^J u)^{1/2}
\\
\leq
&E(s_0, \del^I L^J u)^{1/2}
+
\int_{s_0}^s \Big\| \del^I L^J Q_{u0}(u; v, \del v) + \del^I L^J Q_{u1}(\del u; v, \del v) \Big\|_{L^2_f(\Hcal_{s'})} \, ds'.
\endaligned
$$

We first take into account the cases of $|I| + |J| \leq N-1$, and by the product rule (with two nice functions $f, g$)
$$
\del^I L^J (fg)
=
\sum_{I_1 + I_2 = I, J_1 + J_2 = J} \del^{I_1} L^{J_1} f \, \del^{I_2} L^{J_2} g,
$$
we get
$$
\aligned
&\Big\| \del^I L^J Q_{u0}(u; v, \del v) + \del^I L^J Q_{u1}(\del u; v, \del v) \Big\|_{L^2_f(\Hcal_{s'})}
\\
\lesssim
&\sum_{\substack{ |I_1| + |J_1| \leq N-1\\ |I_2| + |J_2| \leq N-4}} \Big(\big\| (s'/t) \del^{I_1} L^{J_1} u \big\|_{L^2_f(\Hcal_{s'})} + \big\| (s'/t) \del \del^{I_1} L^{J_1} u \big\|_{L^2_f(\Hcal_{s'})} \Big)
\\
&\hskip2.2cm\Big(\big\| (t/s') \del^{I_2} L^{J_2} v \big\|_{L^\infty(\Hcal_{s'})} + \big\| (t/s') \del \del^{I_2} L^{J_2} v \big\|_{L^\infty(\Hcal_{s'})} \Big)
\\
+
&\sum_{\substack{|I_1| + |J_1| \leq N-4\\ |I_2| + |J_2| \leq N-1}} \Big(\big\| \del^{I_1} L^{J_1} u \big\|_{L^\infty(\Hcal_{s'})} + \big\| \del \del^{I_1} L^{J_1} u \big\|_{L^\infty(\Hcal_{s'})} \Big)
\\
&\hskip2.3cm\Big(\big\| \del^{I_2} L^{J_2} v \big\|_{L_f^2(\Hcal_{s'})} + \big\| \del \del^{I_2} L^{J_2} v \big\|_{L_f^2(\Hcal_{s'})} \Big)
\\
\lesssim &(C_1 \eps)^2 s'^{-3/2+\delta},
\endaligned
$$
in which we used (recall $s' \leq t$)
$$
\aligned
&\sum_{|I_2| + |J_2| \leq N-4} \Big(\big\| (t/s') \del^{I_2} L^{J_2} v \big\|_{L^\infty(\Hcal_{s'})} + \big\| (t/s') \del \del^{I_2} L^{J_2} v \big\|_{L^\infty(\Hcal_{s'})} \Big) 
\\
\lesssim 
&\sum_{|I_2| + |J_2| \leq N-4} \Big( \big\| (t/s') t^{-3/2} \big\|_{L^\infty(\Hcal_{s'})} \big\| t^{3/2} \del^{I_2} L^{J_2} v \big\|_{L^\infty(\Hcal_{s'})} 
\\
&\hskip1.7cm+ \big\| (t/s') t^{-3/2} \big\|_{L^\infty(\Hcal_{s'})} \big\| t^{3/2} \del \del^{I_2} L^{J_2} v \big\|_{L^\infty(\Hcal_{s'})} \Big)
\\
\lesssim & C_1 \eps s'^{-3/2},
\endaligned
$$
which yields that
$$
\aligned
E(s, \del^I L^J u)^{1/2}
\lesssim
&E(s_0, \del^I L^J u)^{1/2}
+
(C_1 \eps)^2 \int_{s_0}^s s'^{-3/2+\delta} \, ds'
\\
\lesssim
& \eps + (C_1 \eps)^2.
\endaligned
$$

Similarly, for the cases of $|I| + |J| =N$, we find that 
$$
\aligned
&\Big\| \del^I L^J Q_{u0}(u; v, \del v) + \del^I L^J Q_{u1}(\del u; v, \del v) \Big\|_{L^2_f(\Hcal_{s'})}
\\
\lesssim
&\sum_{\substack{ |I_1| + |J_1| \leq N\\ |I_2| + |J_2| \leq N-4}} \Big(\big\| (s'/t) \del^{I_1} L^{J_1} u \big\|_{L^2_f(\Hcal_{s'})} + \big\| (s'/t) \del \del^{I_1} L^{J_1} u \big\|_{L^2_f(\Hcal_{s'})} \Big)
\\
&\hskip2cm\Big(\big\| (t/s') \del^{I_2} L^{J_2} v \big\|_{L^\infty(\Hcal_{s'})} + \big\| (t/s') \del \del^{I_2} L^{J_2} v \big\|_{L^\infty(\Hcal_{s'})} \Big)
\\
+
&\sum_{\substack{|I_1| + |J_1| \leq N-4\\ |I_2| + |J_2| \leq N}} \Big(\big\| (t/s') \del^{I_1} L^{J_1} u \big\|_{L^\infty(\Hcal_{s'})} + \big\| (t/s') \del \del^{I_1} L^{J_1} u \big\|_{L^\infty(\Hcal_{s'})} \Big)
\\
&\hskip2.8cm\Big(\big\| \del^{I_2} L^{J_2} v \big\|_{L_f^2(\Hcal_{s'})} + \big\| (s'/t) \del \del^{I_2} L^{J_2} v \big\|_{L_f^2(\Hcal_{s'})} \Big).
\endaligned
$$
We pay attention to the term $\big\| (s'/t) \del^{I_1} L^{J_1} u \big\|_{L^2_f(\Hcal_{s'})}$ with $|I_1| + |J_1| \leq N$, and we bound it by distinguishing two situations: $|I_1| = 0$ and $|I_1| \neq 0$. When $|I_1| \neq 0$, with $|I_2| \geq 0$ and $|I_2| + |J_1| \leq N-1$, we have
$$
\big\| (s'/t) \del^{I_1} L^{J_1} u \big\|_{L^2_f(\Hcal_{s'})}
\lesssim \big\| (s'/t) \del \del^{I_2} L^{J_1} u \big\|_{L^2_f(\Hcal_{s'})}
\lesssim C_1 \eps.
$$
When $|I_1| = 0$, the conformal energy estimates in Proposition \ref{prop:BA2} imply
$$
\big\| (s'/t) \del^{I_1} L^{J_1} u \big\|_{L^2_f(\Hcal_{s'})}
\lesssim \big\| (s'/t) L L^{N-1} u \big\|_{L^2_f(\Hcal_{s'})}
\lesssim C_1 \eps s'^{1/2+\delta}.
$$
Thus we obtain
$$
\Big\| \del^I L^J Q_{u0}(u; v, \del v) + \del^I L^J Q_{u1}(\del u; v, \del v) \Big\|_{L^2_f(\Hcal_{s'})}
\lesssim (C_1 \eps)^2 s'^{-1+\delta},
$$
and this gives us
$$
E(s, \del^I L^J u)^{1/2}
\lesssim
\eps + (C_1 \eps)^2 s^\delta.
$$

We thus complete the proof.
\end{proof}

Next, we turn to prove the refined estimates for the conformal energy of $\del^I L^J u$ with $|I| + |J| \leq N-1$,
but this does not seem to be possible with those seemingly bad nonlinearities $Q_{u0}(u; v, \del v) + Q_{u1}(\del u; v, \del v)$. Fortunately, we find that we can reveal a null structure by doing a transformation on the wave equation. And then the $L^2$--type estimates involving the scaling vector field $L_0$ in Lemma \ref{lem:L2type} provide us with very good bounds of the null forms.

In order to achieve our goal, we first do a transformation on the wave equation in \eqref{eq:model}.

\begin{lemma}
Let 
$$
U = u + Q_{u0}(u; v, \del v) + Q_{u1}(\del u; v, \del v),
$$ 
then $U$ solves the following wave equation
\bel{eq:U-eq} 
\aligned
- \Box U 
=
N(u, v)
+ H(u, v),
\\
\big(U, \del_t U \big) (s_0, \cdot)
= \big( U_0, U_1 \big),
\endaligned
\ee
in which $N(u, v)$ represents the quadratic null terms
\be 
N(u, v)
=
- M_u \del_\gamma u \del^\gamma v 
- M_u^\alpha \del_\gamma u \del^\gamma \del_\alpha v 
- N_u^\alpha \del_\gamma \del_\alpha u \del^\gamma v 
- N_u^{\alpha \beta} \del_\gamma \del_\alpha u \del^\gamma \del_\beta v ,
\ee
and $H(u, v)$ represents the negligible cubic terms of the form
\be 
\aligned
H(u, v)
&=
M_u (-\Box u) v + M_u u (-\Box v + v)
+
M_u^\alpha (-\Box u) \del_\alpha v + M_u^\alpha u (-\Box \del_\alpha v + \del_\alpha v)
\\
&+
N_u^\alpha (-\Box \del_\alpha u) v + N_u^\alpha \del_\alpha u (-\Box v + v)
+
N_u^{\alpha \beta} (-\Box \del_\alpha u) \del_\beta v + N_u^{\alpha \beta} \del_\alpha u (-\Box \del_\beta v + \del_\beta v),
\endaligned
\ee
and the initial data are
\be 
\aligned
&\big( U_0, U_1 \big) 
\\
=
&\big( u + Q_{u0}(u; v, \del v) + Q_{u1}(\del u; v, \del v), \del_t u + \del_t Q_{u0}(u; v, \del v) + \del_t Q_{u1}(\del u; v, \del v) \big) (s_0, \cdot).
\endaligned
\ee
\end{lemma}

Since we work in $\RR^{3+1}$ and there exists at least one factor of Klein-Gordon component in the cubic nonlinearities of $H(u, v)$, we know $H(u, v)$ behaves very nicely, and we will not pay extra attention to the cubic terms in $H(u, v)$. We will next derive estimates for the new variable $U$, and then in turn deduce the wanted estimates for the original unknown $u$ by the simple observation that the difference between $u$ and $U$ contains only quadratic terms.

\begin{lemma}\label{lem:U-est}
The following estimates for $U$ hold
\be 
\aligned
E_{con}(s, \del^I L^J U)^{1/2}
&\lesssim
\eps + (C_1 \eps)^2 s^{1/2 + \delta},
\qquad
&|I| + |J| \leq N-1,
\\
E_{con}(s, \del^I L^J U)^{1/2}
&\lesssim
\eps + (C_1 \eps)^2 s^\delta,
\qquad
&|I| + |J| \leq N-2,
\\
E_{con}(s, \del^I L^J U)^{1/2}
&\lesssim
\eps + (C_1 \eps)^2,
\qquad
&|I| + |J| \leq N-3.
\endaligned
\ee

\end{lemma}

\begin{proof}

We act $\del^I L^J$ with $|I| + |J| \leq N-1$ on the $U$ equation \eqref{eq:U-eq} to obtain
$$
- \Box \del^I L^J U 
=
\del^I L^J N(u, v)
+ \del^I L^J H(u, v),
$$
in which $N(u, v)$ represents the quadratic null terms
$$
N(u, v)
=
- M_u \del_\gamma u \del^\gamma v 
- M_u^\alpha \del_\gamma u \del^\gamma \del_\alpha v 
- N_u^\alpha \del_\gamma \del_\alpha u \del^\gamma v 
- N_u^{\alpha \beta} \del_\gamma \del_\alpha u \del^\gamma \del_\beta v, 
$$
and $H(u, v)$ is the negligible cubic nonlinearities.

By Lemma \ref{lem:commu-eq}, it is not hard to see that
$$
\del^I L^J N(u, v)
=
\sum_{I_1 + I_2 = I, J_1 + J_2 = J} N(\del^{I_1} L^{J_1} u, \del^{I_2} L^{J_2} v ).
$$
Then by Lemma \ref{lem:null-classical}, we have
$$
\aligned
N(\del^{I_1} L^{J_1} u, \del^{I_2} L^{J_2} v )
=
&{1\over t} \Big( M_u \big(\del_t \del^{I_2} L^{J_2} v  L_0 \del^{I_1} L^{J_1} u - L_a \del^{I_2} L^{J_2} v \del^a \del^{I_1} L^{J_1} u \big)
\\
&+ M_u^\alpha \big( \del_t \del^{I_2} L^{J_2} \del_\alpha v  L_0 \del^{I_1} L^{J_1} u - L_a \del^{I_2} L^{J_2} \del_\alpha v \del^a \del^{I_1} L^{J_1} u \big)
\\
&+ N_u^\alpha \big( \del_t \del^{I_2} L^{J_2} v  L_0 \del^{I_1} L^{J_1} \del_\alpha u - L_a \del^{I_2} L^{J_2} v \del^a \del^{I_1} L^{J_1} \del_\alpha u \big)
\\
&+N_u^{\alpha \beta} \big( \del_t \del^{I_2} L^{J_2} \del_\beta v  L_0 \del^{I_1} L^{J_1} \del_\alpha u - L_a \del^{I_2} L^{J_2} \del_\beta v \del^a \del^{I_1} L^{J_1} \del_\alpha u \big)
 \Big).
\endaligned
$$
We are now ready to estimate 
$$
\big\| N(\del^{I_1} L^{J_1} u, \del^{I_2} L^{J_2} v ) \big\|_{L^2_f(\Hcal_{s'})},
$$
and we proceed by considering the cases of $|I_2| + |J_2 | \leq N-3$
$$
\aligned
&s'\big\| N(\del^{I_1} L^{J_1} u, \del^{I_2} L^{J_2} v ) \big\|_{L^2_f(\Hcal_{s'})}
\\
\lesssim 
&\big\| \del_t \del^{I_2} L^{J_2} v \big\|_{L^\infty(\Hcal_{s'})} \big\| (s'/t) L_0 \del^{I_1} L^{J_1} u \big\|_{L^2_f(\Hcal_{s'})} + \big\| L_a \del^{I_2} L^{J_2} v \big\|_{L^\infty(\Hcal_{s'})} \big\| (s'/t) \del^a \del^{I_1} L^{J_1} u \big\|_{L^2_f(\Hcal_{s'})}
\\
+ & \big\| \del_t \del^{I_2} L^{J_2} \del v \big\|_{L^\infty(\Hcal_{s'})} \big\| (s'/t) L_0 \del^{I_1} L^{J_1} u \big\|_{L^2_f(\Hcal_{s'})} + \big\| L_a \del^{I_2} L^{J_2} \del v \big\|_{L^\infty(\Hcal_{s'})} \big\| (s'/t) \del^a \del^{I_1} L^{J_1} u \big\|_{L^2_f(\Hcal_{s'})}
\\
+ & \big\| \del_t \del^{I_2} L^{J_2} v \big\|_{L^\infty(\Hcal_{s'})} \big\| (s'/t) L_0 \del^{I_1} L^{J_1} \del u \big\|_{L^2_f(\Hcal_{s'})} + \big\| L_a \del^{I_2} L^{J_2} v \big\|_{L^\infty(\Hcal_{s'})} \big\| (s'/t) \del^a \del^{I_1} L^{J_1} \del u \big\|_{L^2_f(\Hcal_{s'})}
\\
+ & \big\| \del_t \del^{I_2} L^{J_2} \del v\big\|_{L^\infty(\Hcal_{s'})} \big\| (s'/t)  L_0 \del^{I_1} L^{J_1} \del u \big\|_{L^2_f(\Hcal_{s'})} + \big\| L_a \del^{I_2} L^{J_2} \del v \big\|_{L^\infty(\Hcal_{s'})} \big\| (s'/t) \del^a \del^{I_1} L^{J_1} \del u \big\|_{L^2_f(\Hcal_{s'})}.
\endaligned
$$
Similarly, we have for the cases of $|I_1| + |J_1 | \leq N-5$ that
$$
\aligned
&s'\big\| N(\del^{I_1} L^{J_1} u, \del^{I_2} L^{J_2} v ) \big\|_{L^2_f(\Hcal_{s'})}
\\
\lesssim 
&\big\| \del_t \del^{I_2} L^{J_2} v \big\|_{L^2_f(\Hcal_{s'})} \big\| (s'/t) L_0 \del^{I_1} L^{J_1} u \big\|_{L^\infty(\Hcal_{s'})} + \big\| L_a \del^{I_2} L^{J_2} v \big\|_{L^2_f(\Hcal_{s'})} \big\| (s'/t) \del^a \del^{I_1} L^{J_1} u \big\|_{L^\infty(\Hcal_{s'})}
\\
+ & \big\| (s'/t) \del_t \del^{I_2} L^{J_2} \del v \big\|_{L^2_f(\Hcal_{s'})} \big\| L_0 \del^{I_1} L^{J_1} u \big\|_{L^\infty(\Hcal_{s'})} + \big\| (s'/t) L_a \del^{I_2} L^{J_2} \del v \big\|_{L^2_f(\Hcal_{s'})} \big\| \del^a \del^{I_1} L^{J_1} u \big\|_{L^\infty(\Hcal_{s'})}
\\
+ & \big\| \del_t \del^{I_2} L^{J_2} v \big\|_{L^2_f(\Hcal_{s'})} \big\| (s'/t) L_0 \del^{I_1} L^{J_1} \del u \big\|_{L^\infty(\Hcal_{s'})} + \big\| L_a \del^{I_2} L^{J_2} v \big\|_{L^2_f(\Hcal_{s'})} \big\| (s'/t) \del^a \del^{I_1} L^{J_1} \del u \big\|_{L^\infty(\Hcal_{s'})}
\\
+ & \big\| (s'/t) \del_t \del^{I_2} L^{J_2} \del v\big\|_{L^2_f(\Hcal_{s'})} \big\| L_0 \del^{I_1} L^{J_1} \del u \big\|_{L^\infty(\Hcal_{s'})} + \big\| (s'/t) L_a \del^{I_2} L^{J_2} \del v \big\|_{L^2_f(\Hcal_{s'})} \big\| \del^a \del^{I_1} L^{J_1} \del u \big\|_{L^\infty(\Hcal_{s'})}.
\endaligned
$$

By inserting the estimates in Proposition \ref{prop:BA1} and Proposition \ref{prop:BA2}, with some cumbersome but straightforward calculation, we obtain
\begin{eqnarray}
& s' \big\| N(\del^{I_1} L^{J_1} u, \del^{I_2} L^{J_2} v ) \big\|_{L^2_f(\Hcal_{s'})}  \hskip4.5cm  \notag
\\
&\lesssim 
\left\{
\begin{array}{lll}
&(C_1 \eps)^2 s'^{-1/2+\delta}, \qquad |I| + |J| \leq N-1, I_1+I_2 = I, J_1+J_2 = J,
\\
&(C_1 \eps)^2 s'^{-1+\delta}, \qquad |I| + |J| \leq N-2, I_1+I_2 = I, J_1+J_2 = J,
\\
&(C_1 \eps)^2 s'^{-3/2+\delta}, \qquad |I| + |J| \leq N-3, I_1+I_2 = I, J_1+J_2 = J,
\end{array}
\right.
\end{eqnarray}
Easily we can also get such estimates for the negligible cubic term $\del^I L^J H(u, v)$ which read
$$
s' \big\| \del^I L^J H(u, v) \big\|_{L^2_f(\Hcal_{s'})}
\lesssim (C_1 \eps)^3 s'^{-3/2+\delta}.
$$

Finally we recall the conformal energy estimate
$$
\aligned
& E_{con}(s, \del^I L^J U)^{1/2} 
\\
\leq
& E_{con}(s_0, \del^I L^J U)^{1/2} 
+
\int_{s_0}^s s' \Big\| \del^I L^J N(u, v) + \del^I L^J H(u, v) \Big\|_{L^2_f(\Hcal_{s'})} \, ds',
\endaligned
$$
which implies 
\begin{eqnarray}
 E_{con}(s, \del^I L^J U)^{1/2} 
\lesssim 
&\left\{
\begin{array}{lll}
&\eps + (C_1 \eps)^2 s^{1/2 + \delta}, &\qquad |I| + |J| \leq N-1,  \notag
\\
&\eps + (C_1 \eps)^2 s^\delta, &\qquad |I| + |J| \leq N-2,
\\
&\eps + (C_1 \eps)^2, &\qquad |I| + |J| \leq N-3, \notag
\end{array}
\right.
\end{eqnarray}
and this completes the proof.
\end{proof}

\begin{proposition}\label{prop:u-con2}
The following improved estimates are valid
\be 
\aligned
E_{con}(s, \del^I L^J u)^{1/2}
&\lesssim
\eps + (C_1 \eps)^2 s^{1/2+\delta},
\qquad
&|I| + |J| \leq N-1,
\\
E_{con}(s, \del^I L^J u)^{1/2}
&\lesssim
\eps + (C_1 \eps)^2 s^\delta,
\qquad
&|I| + |J| \leq N-2,
\\
E_{con}(s, \del^I L^J u)^{1/2}
&\lesssim
\eps + (C_1 \eps)^2,
\qquad
&|I| + |J| \leq N-3.
\endaligned
\ee

\end{proposition}

\begin{proof}

Recall the relation between $u$ and $U$ reads
$$
U = u + Q_{u0}(u; v, \del v) + Q_{u1}(\del u; v, \del v), 
$$
hence we are led to
$$
\aligned
&E_{con}(s, \del^I L^J u)^{1/2}
\\
\lesssim
&E_{con}(s, \del^I L^J U)^{1/2}
+
E_{con}(s, \del^I L^J Q_{u0}(u; v, \del v) )^{1/2}
+
E_{con}(s, \del^I L^J Q_{u1}(u; v, \del v) )^{1/2}.
\endaligned
$$
Since $Q_{u0}(u; v, \del v), Q_{u1}(u; v, \del v)$ are quadratic terms, we easily obtain
\begin{eqnarray}
& \hskip3.6cm E_{con}(s, \del^I L^J Q_{u0}(u; v, \del v) )^{1/2}
+
E_{con}(s, \del^I L^J Q_{u1}(u; v, \del v) )^{1/2}  \hskip1cm  \notag
\\
&\lesssim  
\left\{
\begin{array}{lll}
&(C_1 \eps)^2 s^{2\delta}, \hskip2cm & |I| + |J| \leq N-1,
\\
& (C_1 \eps)^2,  \hskip2cm & |I| + |J| \leq N-2,
\end{array}
\right.
\end{eqnarray}
which finishes the proof.

\end{proof}


\subsection{Improved estimates for the Klein-Gordon component $v$}

We now turn to improve the estimates for the Klein-Gordon component $v$, which are more straightforward.

\begin{proposition}\label{prop:v-E}
It holds
\bel{eq:KG-improved}
\aligned
E_1(s, \del^I L^J v)^{1/2}
&\lesssim
\eps + (C_1 \eps)^2 s^\delta,
\qquad
 |I| + |J| = N,
\\
E_1(s, \del^I L^J v)^{1/2}
&\lesssim
\eps + (C_1 \eps)^2,
\qquad
 |I| + |J| \leq N-1.
\endaligned
\ee
\end{proposition}

\begin{proof}

We act the operator $\del^I L^J$ to the Klein-Gordon equation in \eqref{eq:model} to have
$$
-\Box \del^I L^J v + \del^I L^J v 
= \del^I L^J Q_{v0}(u; v, \del v) + \del^I L^J Q_{v1}(\del u; v, \del v).
$$
By the energy estimates we have
$$
\aligned
&E_1(s, \del^I L^J v)^{1/2}
\\
\leq
&E_1(s_0, \del^I L^J v)^{1/2}
+
\int_{s_0}^s \Big\| \del^I L^J Q_{v0}(u; v, \del v) + \del^I L^J Q_{v1}(\del u; v, \del v) \Big\|_{L^2_f(\Hcal_{s'})} \, ds'.
\endaligned
$$

We then start with the cases of $|I| + |J| \leq N-1$, and it is easy to find that
$$
\aligned
&\Big\| \del^I L^J Q_{v0}(u; v, \del v) + \del^I L^J Q_{v1}(\del u; v, \del v) \Big\|_{L^2_f(\Hcal_{s'})}
\\
\lesssim
&\sum_{\substack{ |I_1| + |J_1| \leq N-1\\ |I_2| + |J_2| \leq N-4}} \Big(\big\| (s'/t) \del^{I_1} L^{J_1} u \big\|_{L^2_f(\Hcal_{s'})} + \big\| (s'/t) \del \del^{I_1} L^{J_1} u \big\|_{L^2_f(\Hcal_{s'})} \Big)
\\
&\hskip2cm\Big(\big\| (t/s') \del^{I_2} L^{J_2} v \big\|_{L^\infty(\Hcal_{s'})} + \big\| (t/s') \del \del^{I_2} L^{J_2} v \big\|_{L^\infty(\Hcal_{s'})} \Big)
\\
+
&\sum_{\substack{|I_1| + |J_1| \leq N-4\\ |I_2| + |J_2| \leq N-1}} \Big(\big\| \del^{I_1} L^{J_1} u \big\|_{L^\infty(\Hcal_{s'})} + \big\| \del \del^{I_1} L^{J_1} u \big\|_{L^\infty(\Hcal_{s'})} \Big)
\\
&\hskip2.2cm\Big(\big\| \del^{I_2} L^{J_2} v \big\|_{L_f^2(\Hcal_{s'})} + \big\| \del \del^{I_2} L^{J_2} v \big\|_{L_f^2(\Hcal_{s'})} \Big),
\endaligned
$$
and then by recalling $s' \leq t$ and noting
$$
\aligned
&\sum_{|I_2| + |J_2| \leq N-4} \Big(\big\| (t/s') \del^{I_2} L^{J_2} v \big\|_{L^\infty(\Hcal_{s'})} + \big\| (t/s') \del \del^{I_2} L^{J_2} v \big\|_{L^\infty(\Hcal_{s'})} \Big) 
\\
\lesssim 
&\sum_{|I_2| + |J_2| \leq N-4} \Big( \big\| (t/s') t^{-3/2} \big\|_{L^\infty(\Hcal_{s'})} \big\| t^{3/2} \del^{I_2} L^{J_2} v \big\|_{L^\infty(\Hcal_{s'})} 
\\
&\hskip1.7cm+ \big\| (t/s') t^{-3/2} \big\|_{L^\infty(\Hcal_{s'})} \big\| t^{3/2} \del \del^{I_2} L^{J_2} v \big\|_{L^\infty(\Hcal_{s'})} \Big)
\\
\lesssim & C_1 \eps s'^{-3/2},
\endaligned
$$
we further have
$$
\aligned
\Big\| \del^I L^J Q_{v0}(u; v, \del v) + \del^I L^J Q_{v1}(\del u; v, \del v) \Big\|_{L^2_f(\Hcal_{s'})}
\lesssim (C_1 \eps)^2 s'^{-3/2+\delta},
\endaligned
$$
which is integrable, and this leads us to
$$
E_1(s, \del^I L^J v)^{1/2}
\lesssim
\eps + (C_1 \eps)^2,
\qquad
|I| + |J| \leq N-1.
$$

Next, we consider the cases of $|I| + |J| =N$, and we find similarly that 
$$
\aligned
&\Big\| \del^I L^J Q_{v0}(u; v, \del v) + \del^I L^J Q_{v1}(\del u; v, \del v) \Big\|_{L^2_f(\Hcal_{s'})}
\\
\lesssim
&\sum_{\substack{ |I_1| + |J_1| \leq N\\ |I_2| + |J_2| \leq N-4}} \Big(\big\| (s'/t) \del^{I_1} L^{J_1} u \big\|_{L^2_f(\Hcal_{s'})} + \big\| (s'/t) \del \del^{I_1} L^{J_1} u \big\|_{L^2_f(\Hcal_{s'})} \Big)
\\
&\hskip2cm\Big(\big\| (t/s') \del^{I_2} L^{J_2} v \big\|_{L^\infty(\Hcal_{s'})} + \big\| (t/s') \del \del^{I_2} L^{J_2} v \big\|_{L^\infty(\Hcal_{s'})} \Big)
\\
+
&\sum_{\substack{|I_1| + |J_1| \leq N-4\\ |I_2| + |J_2| \leq N}} \Big(\big\| (t/s') \del^{I_1} L^{J_1} u \big\|_{L^\infty(\Hcal_{s'})} + \big\| (t/s') \del \del^{I_1} L^{J_1} u \big\|_{L^\infty(\Hcal_{s'})} \Big)
\\
&\hskip2.1cm\Big(\big\| \del^{I_2} L^{J_2} v \big\|_{L_f^2(\Hcal_{s'})} + \big\| (s'/t) \del \del^{I_2} L^{J_2} v \big\|_{L_f^2(\Hcal_{s'})} \Big)
\\
\lesssim &(C_1 \eps)^2 s'^{-1+\delta},
\endaligned
$$
and this yields
$$
E_1(s, \del^I L^J v)^{1/2}
\lesssim
\eps + (C_1 \eps)^2 s^\delta,
\qquad
|I| + |J| \leq N.
$$
Hence the proof is done.

\end{proof}


\subsection{Proof of Theorem \ref{thm:wKG}}

By gathering the refined estimates we derived above, we are ready to give the proof of Theorem \ref{thm:wKG}.

\begin{proof}[Proof of Theorem \ref{thm:wKG}]
Recall the estimates obtained in propositions \ref{prop:u-E}, \ref{prop:u-con2}, and \ref{prop:v-E}, and we first choose $C_1$ very large so that $\lesssim \eps$ gives $\leq {1\over 4} C_1 \eps$, and then choose $\eps$ sufficiently small so that $\lesssim (C_1 \eps)^2$ gives $\leq {1\over 4} C_1 \eps$. By this choice, we easily have all of the wanted refined estimates in \eqref{eq:BA-refined}. Thus if $s_0 < T < + \infty$ is some finite number, then we can extend it a little bit to the future, but this contradicts its definition. This means $T$ must be $+\infty$ which gives us the global existence result. The sharp pointwise decay results in \eqref{eq:thm} immediately follows the estimates in Section \ref{subsec:BA}.

\end{proof}


\section{Acknowledgments} This project has received funding from the European Research Council (ERC) under the European Union's Horizon 2020 research and innovation program Grant agreement No 637653, project BLOC "Mathematical Study of Boundary Layers in Oceanic Motion". The author would also like to acknowledge the support from the Innovative Training Networks (ITN) grant 642768, which is entitled ModCompShock.




\end{document}